\documentclass[10pt]{amsart}

\usepackage[foot]{amsaddr}
\usepackage[english]{babel}
\usepackage[
]{hyperref}
\usepackage[usenames,dvipsnames,svgnames,table,x11names]{xcolor}
\usepackage{pgfplots}
\usepackage{tikz}
\usepackage{tikz-3dplot}
\usetikzlibrary{calc,shadows,shapes.callouts,shapes.geometric,shapes.misc,positioning,patterns,decorations.pathmorphing,decorations.markings,decorations.fractals,decorations.pathreplacing,shadings,fadings}

\usepackage[utf8]{inputenc}

\usepackage[total={16cm,20cm}, top=4.5cm, left=2.5cm]{geometry}
\usepackage{rotating,multirow,pdflscape}
\usepackage{amssymb,latexsym,amsmath,amsthm}
\usepackage{mathrsfs}
\usepackage{mathtools,arydshln,mathdots}
\usepackage{bigints}
\usepackage{comment}
\usepackage{wrapfig}

\makeatletter
\renewcommand*\env@matrix[1][*\c@MaxMatrixCols c]{%
 \hskip - 0.7\arraycolsep
 \let\@ifnextchar\new@ifnextchar
 \array{#1}}
\makeatother

\mathtoolsset{centercolon}
\usepackage[table]{xcolor}
\usepackage[toc,page]{appendix}

\usepackage{tocvsec2}

\usepackage{Baskervaldx}
\usepackage{newtxmath }
\newtheorem{defi}{Definition}

\newtheorem{pro}{Proposition}

\renewcommand{\d}{\operatorname{d}}

\newcommand{\diag}{\operatorname{diag}}

\newcommand{\C}{\mathbb{C}}

\newcommand{\A}{\mathbb{A}}

\newcommand{\R}{\mathbb{R}}

\newcommand{\prodint}[1]{\left\langle{#1}\right\rangle}

\makeatletter
\def\@settitle{\begin{center}%
  \baselineskip14\p@\relax
  \bfseries
  \uppercasenonmath\@title
  \@title
  \ifx\@subtitle\@empty\else
  \\[1ex]\uppercasenonmath\@subtitle
  \footnotesize\mdseries\@subtitle
  \fi
 \end{center}%
}
\def\subtitle#1{\gdef\@subtitle{#1}}
\def\@subtitle{}
\makeatother

\title[Orthogonal Polynomials and $LU$ factorization]{Revisiting Biorthogonal Polynomials. \\An $LU$ factorization  discussion}

\author[M Mañas]{Manuel Mañas}
\address{Departamento de Física Teórica, Universidad Complutense de Madrid, 28040-Madrid, Spain \&
	Instituto de Ciencias Matematicas (ICMAT),  Campus de Cantoblanco UAM, 28049-Madrid, Spain}
\email{manuel.manas@ucm.es}

\thanks{\textsc{Acknowledgments:} The author thanks the financial support from the Spanish \emph{Agencia Estatal de Investigación}, research project PGC2018-096504-B-C3  entitled  \emph{Ortogonalidad y Aproximación: Teoría y Aplicaciones en Física Matemática}. The author acknowledges the exhaustive revision by the anonymous referee that  much improved the readability of this contribution}

\begin{document}

\maketitle

\begin{abstract}
The Gauss--Borel or $LU$ factorization  of Gram matrices of bilinear forms is the pivotal element in  the discussion of the theory of biorthogonal polynomials. The construction of biorthogonal families of polynomials and its second kind functions, of the spectral matrices modeling the multiplication by the independent variable $x$, the Christoffel--Darboux kernel and its projection properties, are discussed from this point of view. Then, the Hankel case is presented and different properties, specific of this case, as the three terms relations, Heine formulas, Gauss quadrature and the Christoffel-Darboux formula are given. The classical orthogonal polynomial of Hermite, Laguerre and Jacobi type are discussed and characterized within this scheme. Finally, it  is shown who this approach is instrumental in the derivation of  Christoffel formulas for general Christoffel and Geronimus perturbations of the bilinear forms.
\end{abstract}

\section{Introduction}

These notes correspond to the \emph{orthogonal polynomial} part of the five lectures I delivered during the VII Iberoamerican Workshop in Orthogonal Polynomials and Applications (Seventh EIBPOA). As such, more than presenting new original material, they are intended to give an alternative, but  consistent and systematic, construction of the theory of biorthogonal polynomials using the $LU$ factorization of the Gram matrix of a given bilinear form. 
We refer the interested reader to the  classical texts  \cite{Gautschi2004Orthogonal,Szego1939Orthogonal} for a general background on the subject.

At the beginning, when our group started to work in this area,  this $LU$ approach was motivated by the connection of  the theory of orthogonal polynomials and the theory of integrable systems, and the ubiquity of factorization problems in the description of the later, see \cite{Sato1981Soliton,Date1982transformation,
	Mulase1984Complete,BtK2,Kac2003KP,manas2009multicomponent1,manas2009multicomponent2}.  Adler and van Moerbeke performed a pioneering work regarding this approach \cite{adler-van moerbeke1997,adler-van moerbeke1999,adler-van moerbeke2001,adler-van moerbeke2009}.  Despite this original motivation, we soon realized that this factorization technique allows for a general and systematic approach to the subject of orthogonal polynomials, giving a unified framework that can be extended to more sophisticated orthogonality scenarios.  Given the understandable space constraint of this volume we will not describe the $LU$ description of the Toda/KP integrable hierarchies and refer the reader to the \href{https://docs.google.com/viewer?a=v&pid=sites&srcid=ZGVmYXVsdGRvbWFpbnxlaWJwb2EyMDE4fGd4OjNjYTQ0YmRmNTc1NjhiNTU}{slides}  of my lectures posted in the web page of Seventh EIBPOA.
 
With the background given here we hope that the reader could understand further developments applied in other orthogonality situations. Let us now describe what we have  done regarding this research in with more general orthogonality frameworks.   In \cite{cum1} we studied the \textbf{generalized orthogonal polynomials }\cite{adler-van moerbeke1999} and its matrix extensions from the Gauss--Borel view point. In \cite{cum2} we gave a complete study in terms of this factorization for \textbf{multiple orthogonal polynomials of mixed type }and characterized the integrable systems associated to them. Then, we studied Laurent orthogonal polynomials in the unit circle trough the CMV approach in \cite{carlos} and found in \cite{carlos2} the Christoffel--Darboux formula for generalized orthogonal matrix polynomials. 
 These methods were further extended, for example we gave an alternative Christoffel--Darboux formula for mixed multiple orthogonal polynomials \cite{gerardo1} or developed the corresponding  theory of \textbf{matrix  Laurent orthogonal polynomials in the unit circle }and its associated Toda type hierarchy \cite{MOLPUC}. In \cite{carlos3,geronimus1,geronimus2} a complete analysis, in terms of the spectral theory of matrix polynomials, of Christoffel and Geronimus perturbations and Christoffel formulas was given for \textbf{matrix orthogonal polynomials}, while in \cite{CMV biorthogonal} we gave a complete description for the Christoffel formulas corresponding to Christoffel perturbations for univariate CMV Laurent orthogonal polynomials. We also mention recent developments on \textbf{multivariate orthogonal polynomials }in real spaces (MVOPR), the corresponding Christoffel formula and the interplay with algebraic geometry \cite{MVOPR,MVOPR-darboux}. Similar multivariate situations but for the complex torus and the CMV ordering was analyzed in \cite{MVOLPUT}.
 
 During my lectures I learned, with pleasure, that Diego Dominici is also interested  in the ways  the $LU$ factorization is an useful tool in the theory of orthogonal polynomials. A paper by Dominici related to matrix factorizations and OP will appear in the journal \href{https://www.worldscientific.com/doi/abs/10.1142/S2010326320400031}{Random Matrices: Theory and Applications.}

 \section{$LU$ factorization and Gram matrix}
 
 Given a square complex matrix  $A\in\C^{N\times N}$, an $LU$ factorization refers to the factorization of $A$ into a lower triangular matrix $L$ and an upper triangular matrix $U$

\subsection{$LDU$ factorization}
	An $LDU$ decomposition is a decomposition of the form
	\begin{align*}
	A=LDU,
	\end{align*}
	where $D$ is a diagonal matrix, and $L$ and $U$ are unitriangular matrices, meaning that all the entries on the diagonals of $L$ and $U$ are one.


\paragraph{\textbf{Existence and uniqueness}}
\begin{itemize}
\item Any   $A\in\operatorname{GL}_N(\C)$ admits an $ LU $(or $LDU$) factorization if and only if all its leading principal minors are nonzero.

\item If  $A\in \C^{N\times N}$ is a singular matrix of rank $r$, then it admits an $LU$ factorization if the first  $r$ leading principal minors are nonzero, although the converse is not true.

\item 
If  $A\in \C^{N\times N}$ has an $LDU$ factorization,  then the factorization is unique.
\end{itemize}

\paragraph{\textbf{Cholesky factorization}}
\begin{itemize}
	\item 	If  $A\in\C^{N\times N}$ is a symmetric ($A=A^\top$)  matrix, we can choose $U$ as the transpose of $L$
	\begin{align*}
	A = L DL^\top.
	\end{align*}
	
\item When  $A$ is Hermitian, $A=A^\dagger$,  we can similarly find
	\begin{align*}
	A = L D L^\dagger.
	\end{align*}
\end{itemize}
 These decomposition are called  Cholesky factorizations. The Cholesky decomposition always exists and is unique — provided the matrix is positive definite.

\subsection{Schur complements}
	Given a block matrix $M$ 
	\begin{align*}
	M=\begin{pmatrix}[c|c]
	A & B\\
	\hline
	C & D
	\end{pmatrix}
	\end{align*}
	with $A\in \operatorname{GL}_p(\C)$, $B\in \C^{p\times q}$ , $C\in \C^{q\times p}$  and $D\in \C^{q\times q}$, we define its Schur complement with respect to $A$  
	as
\begin{align*}
	M\backslash A:=D-CA^{-1}B\in\C^{q\times q}.
	\end{align*}

\begin{pro}[Schur complements and $LU$ factorization]
	If $A$ is nonsingular we have for the block matrix $M$ the following factorization
	\begin{align*}
	M=
	\begin{pmatrix}
	I_p & 0_{p\times q}\\
	CA^{-1} & I_q
	\end{pmatrix}
	\begin{pmatrix}
	A & 0_{p\times q}\\
	0_{q\times p}  & M\backslash A
	\end{pmatrix}
	\begin{pmatrix}
	I_p & A^{-1}B\\
	0_{q\times p} & I_q
	\end{pmatrix}.
	\end{align*}
\end{pro}

For a detailed account of Schur complements see \cite{zhang}.

Given a matrix $M=(M_{i,j})_{i,j=0}^{N-1}\in \C^{N\times N}$ a $\ell+1$,  $\ell<N$, truncation is given by
\begin{align*}
M^{[\ell+1]}&:=\begin{pmatrix}
M_{0,0}      & M_{0,1}      & \dots  &  M_{0,\ell}\\
M_{1,0}      & M_{1,1}      & \dots  &  M_{1,\ell} \\
\vdots       & \vdots       & \ddots  &  \vdots \\
M_{\ell,0} & M_{\ell,1} & \dots  & M_{\ell,\ell} 
\end{pmatrix}.
\end{align*}

\begin{pro}[Nonzero minors and $LU$ factorization]
Any  nonsingular matrix $M\in\operatorname{GL}_N(\C)$ with  all its leading principal minors different from zero, i. e., $\det M^{[\ell+1]}\neq 0$, $\ell\in\{0,1,,\dots,  N-1\}$, has a $LDU$ factorization.
\end{pro}

\begin{proof}
	As $\det M^{[N-1]}\neq 0$, we have
\begin{align*}
M=
\begin{pmatrix}
I_{N-1} & 0_{(N-1)\times 1}\\
v_N  &1
\end{pmatrix}
\begin{pmatrix}
M^{[N-1]} & 0_{(N-1)\times 1}\\
0_{1\times (N-1)} & M\backslash M^{[N-1]}
\end{pmatrix}
\begin{pmatrix}
I_{N-1} & w_N\\
0_{1\times (N-1)} & 1
\end{pmatrix}
\end{align*}
where $v_N$ and $w_N$ are  row and column vectors, respectively,  with $N-1$ components.

Now, from $\det M^{[N-2]}\neq 0$ we deduce
{\small		\begin{align*}\hspace*{-0.6cm}
	M^{[N-1]}=
	\begin{pmatrix}
	I_{N-2} & 0_{(N-2)\times 1}\\
	v_{N-1}  &1
	\end{pmatrix}
	\begin{pmatrix}
	M^{[N-2]} & 0_{(N-2)\times 1}\\
	0_{1\times (N-2)} & M^{[N-1]}\backslash M^{[N-2]}
	\end{pmatrix}
	\begin{pmatrix}
	I_{N-2} & w_{N-1}\\
	0_{1\times (N-2)} & 1
	\end{pmatrix}
	\end{align*}}
which inserted in the previous result yields
\begin{align*}
	M=
	\begin{pmatrix}[c|cc]
	I_{N-2} & \multicolumn{2}{c}{0_{(N-2)\times 2}}\\\hline
	\multirow{2}{*}{*} &1 &0\\
	&*&1
	\end{pmatrix}
	\begin{pmatrix}[c|cc]
	M^{[N-2]} & 
	\multicolumn{2}{c}{0_{(N-2)\times 2}}\\\hline
	\multirow{2}{*}{$0_{2\times (N-2)}$}& M^{[N-1]}\backslash M^{[N-2]}&0\\
	&0 & M\backslash M^{[N-1]}
	\end{pmatrix}
	\begin{pmatrix}[c|cc]
	I_{N-2} & \multicolumn{2}{c}{*}\\\hline
	\multirow{2}{*}{$0_{2\times (N-2)} $}& 1 & *\\
	&0 &1
	\end{pmatrix}.
 	\end{align*}
We finally get an $LDU$ factorization with 
\begin{align*}
D=\diag\big(
M^{[1]} \backslash M^{[0]}, M^{[2]} \backslash M^{[1]},\dots,
M\backslash M^{[N-1]}\big).
\end{align*}
\end{proof}
\subsection{Bilinear forms, Gram matrices and $LU$ factorizations}

	A  bilinear  form  $\prodint{\cdot,\cdot}$  on the ring of complex valued polynomials in the variable $x$,  $\mathbb{C}[x]$,  is a continuous map
	\begin{align*}
	\begin{array}{cccc}
	\prodint{\cdot,\cdot}: &\mathbb{C}[x]\times\mathbb{C}[x]&\longrightarrow &\mathbb{C}\\
	&(P(x), Q(x))&\mapsto& \prodint{P(x),Q(y)}
	\end{array}
	\end{align*}
	such that for any triple $P(x),Q(x),R(x)\in  \mathbb{C}[x]$ the following properties are fulfilled
	\begin{enumerate}
		\item  $\prodint{AP(x)+BQ(x),R(y)}=A\prodint{P(x),R(y)}+B\prodint{Q(x),R(y)}$, $\forall A,B\in\mathbb{C}$,
		\item $\prodint{P(x),AQ(y)+BR(y)}=\prodint{P(x),Q(y)}A+\prodint{P(x),R(y)}B$, $\forall A,B\in\mathbb{C}$.
	\end{enumerate}

Observe that we have  not chosen the   conjugate in one of the variables.

For  $P(x)=\sum\limits_{k=0}^{\deg P}p_kx^k$ and $Q(x)=\sum\limits_{l=0}^{\deg Q} q_lx^l$ the bilinear form is 
\begin{align*}
\prodint{P(x),Q(y)}&=\sum_{\substack{k=1,\dots,\deg P\\
	l=1,\dots,\deg Q}}p_k G_{k,l}q_l, &
G_{k,l}&=\prodint{x^k ,y^l }.
\end{align*}
The corresponding \textbf{semi-infinite} matrix
\begin{align*}
G=\begin{pmatrix}
G_{0,0 } &G_{0,1}& \dots\\
G_{1,0} & G_{1,1} & \dots\\
\vdots & \vdots
\end{pmatrix},
\end{align*}
is  the  so called \textbf{ Gram matrix} of the sesquilinear form.

\paragraph{\textbf{Examples:}}
\begin{itemize}
	\item \textbf{Borel measures.}
	A first example is given by a complex (or real) Borel measure $\d\mu$ with support  $\operatorname{supp}(\d\mu)\subset\mathbb R$ . 
	Given any pair of matrix polynomials $P(x),Q(x)\in\C[x]$  we introduce the following    bilinear form
	\begin{align*}
	\prodint{P(x),Q(x)}_\mu=\int_\R P(x)\d\mu(x)Q(x).
	\end{align*}
\item \textbf{Example: Linear functionals}
	{\small We consider the space of polynomials $\mathbb C[x]$, with an appropriate topology,  as the space of fundamental functions and take the space of generalized functions as the corresponding continuous linear functionals.  Take a linear functional $u\in (\C[x])'$ and consider}
	\begin{align*}
	\prodint{P(x),Q(x)}_u=u( P(x)Q(x)).
	\end{align*}
\end{itemize}

In both examples the Gram matrix is a Hankel matrix $G_{i+1,j}=G_{i,j+1}$
In these cases,  the Gram matrix is also known as moment matrix as we have 
\begin{align*}
G_{i,j}=\int_{\R} x^{i+j}\d \mu(x)=m_{i+j},
\end{align*}
where $m_j$ is known as the $j$-moment, 
for the measure case, while for the linear functional scenario we have
\begin{align*}
G_{i,j}=u( x^{i+j}).
\end{align*}

\paragraph{\textbf{Schwartz generalized kernels}}
	There are bilinear forms which do not have this particular Hankel type symmetry.  Let    $u_{x,y}\in (\mathbb C[x,y])'\cong\mathbb C[\![x,y]\!]$ so that
	\begin{align*}
	\prodint{P(x),Q(y)}=\langle u_{x,y}, P(x)\otimes Q(y)\rangle.
	\end{align*}
	The Gram matrix of this bilinear form has entries 
	\begin{align*}
	G_{k,l}=\langle u_{x,y}, x^k\otimes y^{l}\rangle. 
	\end{align*}
	This gives a continuous  linear map $\mathcal L_u: \mathbb C[y]\to(\mathbb C [x])'$ such that $\prodint{P(x),Q(y)}=\langle \mathcal L_u(Q(y)), P(x)\rangle$.
	See \cite{Schwartz1} for an introduction to bivariate generalized kernels.

\textbf{Integrals kernels.}
A kernel  $u(x,y)$ is a complex valued locally integrable function,
		that defines an integral operator $f(x)\mapsto g(x)=\int u(x,y) f(y)\d y $ and
		\begin{align*}
		\prodint{P(x),Q(y)}=\int_{\R^2} P(x) u(x,y)		Q(y)\d x\d y 
		\end{align*}

	There is an obvious way of ordering the monomials $\{x^n\}_{n=0}^\infty$ in a semi-infinite vector
	\begin{align*}
	\chi(x)=(1, x,x^2,\dots)^\top
	\end{align*}
 Also, we consider
	\begin{align*}
	\chi^*(x)& =
(x^{-1}, x^{-2},x^{-3},\dots)^\top, &\text{with}& & (\chi(x))^\top \chi^*(y)&=\frac{1}{x-y}, & |y|<|x|.
	\end{align*}

The semi-infinite Gram matrix can be written as follows
$G=\prodint{\chi, \chi^\top}$.
For a Borel measure it reads
$G=\int \chi(x) \big(\chi(x)\big)^\top\d\mu(x)$
and for a linear functional
$G=\prodint{u(x),\chi(x) \big(\chi(x)\big)^\top}$.
When dealing with an integral kernel we have
$
G=\int \chi(x) \big(\chi(y)\big)^\top u(x,y)\d x$
and for a Schwartz kernel
$G=\prodint{u(x,y),\chi(x) \big(\chi(y)\big)^\top}$.

\section{Orthogonal polynomials}

\begin{defi}[Quasi-definite bilinear forms]
	We say that a bilinear form $\prodint{\cdot,\cdot}$ is quasi-definite whenever its Gram matrix 
	has all its leading principal minors different from zero.
\end{defi}

\begin{pro}[Quasi-definiteness and $LDU$ factorization]
	The Gram matrix of a quasi-definite bilinear form admits a unique $LDU$ factorization.
\end{pro}

	Given a quasi-definite bilinear form in the space of polynomials we consider the  $LDU$ factorization of its Gram matrix in the form 
	\begin{align*}
	G=(S_1)^{-1} H (S_2)^{-\top}
	\end{align*}
	where $S_1$ and $S_2$ are lower unitriangular matrices and $H$ is a diagonal matrix.

When the quasi-definite bilinear form comes from a Borel measure or a linear functional the corresponding Gram matrix, now known as moment matrix, is symmetric: $G=G^\top$. Thus, the $LDU$ factorization becomes a Cholesky factorization. Whenever  the Borel measure is positive  (sign defined will equally do) the moment matrix is a positive definite  matrix, i.e.,  all  the principal minors of the moment matrix are strictly positive.
Given either a  Borel measure or a linear functional,  we consider its $LDU$ factorization 
\begin{align*}
G=S^{-1} H S^{-\top},
\end{align*}
where $S$ is a  lower unitriangular matrix and $H$ is a diagonal matrix. For a Borel positive measure, the diagonal coefficients  $H_k$ are positive, $H_k>0$.

\begin{defi}[Constructing the polynomials]
	Given a Gram matrix and its Gauss--Borel factorization we construct the following two families of polynomials
\begin{align*}
P_1(x)&:= S_1\chi(x), & P_2(x)&:= S_2\chi(x).
\end{align*}
Here $P_i=(P_{i,0},P_{i,1},\dots, )$ with $P_{i,k}(x)=x^k+\cdots$, monic polynomials of  degree $k$.
\end{defi}

\begin{pro}[Orthogonality relations]
The above families of polynomials  satisfy the following orthogonality relations
\begin{align*}
\prodint{P_{1,k}(x),y^l }&=\delta_{l,k}H_k, & 
\prodint{x^l, P_{2,k}(y) }&=\delta_{l,k}H_k, & 0\leq l<k.
\end{align*}
\end{pro}

\begin{proof}
	We have that $\prodint{P_1(x), \big(\chi(y)\big)^\top}=S_1\prodint{\chi(x),(\chi(y))^\top}=S_1G=
H(S_2)^{-\top}$ is an upper triangular matrix, and the result follows.
\end{proof}

\begin{pro}[Biorthogonality relations]

The above families of polynomials are biorthogonal
\begin{align*}
\prodint{P_{1,k}(x),P_{2,l}(x) }&=\delta_{l,k}H_k.
\end{align*}
\end{pro}
\begin{proof}
	We have that $\prodint{P_1(x), \big(P_2(y)\big)^\top}=
	S_1\prodint{\chi(x),(\chi(y))^\top}(S_2)^{\top}=S_1G=H$.
\end{proof}

\subsection{Quasi-determinants}
We include this brief  section here because, despite the fact that for the standard orthogonality the quotient of determinants is enough to describe adequately  the results, in more general situations quasi-determinants are needed. Also, even in this situation they give more compact expressions. As we will see we can understand them as an extension of determinants to noncommutative situations and also as Schur complements.
\paragraph{\textbf{Some history}}
	In the late 1920  Archibald Richardson, one of the two responsible of Littlewood--Richardson rule,  and the famous logician Arend Heyting, founder of intuitionist logic, studied possible extensions of the determinant notion to division rings. Heyting defined the \emph{designant} of a matrix with noncommutative entries, which for $2\times 2$ matrices was the Schur complement, and generalized to larger dimensions by induction.
\paragraph{\textbf{The situation nowadays}}
	1990 till today, was given by  Gel'fand, Rektah and collaborators, see \cite{Gelfand2005}.
	Quasi-determinants were defined over free division rings and was early noticed that\textbf{ is not an analog of the commutative determinant but rather of a ratio of determinants}. A cornerstone for  quasi-determinants is  the  \emph{heredity principle}, \textbf{quasi-determinants of quasi-determinants are quasi-determinants}; there is no analog of such a principle for determinants.

\subsubsection{The easiest quasi-determinant: a Schur complement}  We start with $k=2$, so that $A=\begin{psmallmatrix}  A_{1,1} & A_{1,2}\\ A_{2,1} & A_{2,2} \end{psmallmatrix}$. In this case the first quasi-determinant
$\Theta_1(A)\coloneq A/ A_{1,1}$; i. e., a Schur complement which requires $\det A_{1,1}\neq 0$
\paragraph{\textbf{Olver vs Gel'fand}} 
The notation of Olver \cite{Olver} 
is different to that of the Gel'fand school were $\Theta_1(A)=|A|_{2,2}=\begin{vsmallmatrix}  A_{1,1} & A_{1,2}\\ A_{2,1} & {\boxed{ A_{2,2}}} \end{vsmallmatrix}$.
There is another quasi-determinant $\Theta_2(A)=A/ A_{22}=|A|_{1,1}=\begin{vsmallmatrix} {\boxed{A_{1,1}}} & A_{1,2}\\ A_{2,1} & A_{2,2} \end{vsmallmatrix}$, the other Schur complement, and we need $A_{2,2}$ to be a non singular matrix. Other quasi-determinants that can be considered for regular square blocks are
$\begin{vsmallmatrix}
A_{1,1} &A_{1,2}\\
{\boxed{A_{2,1}}}&A_{2,2}
\end{vsmallmatrix}
$
and $\begin{vsmallmatrix} A_{1,1} &{\boxed{A_{1,2}}}\\ A_{2,1} & A_{2,2} \end{vsmallmatrix}$. These last two quasi-determinants are out of the scope of these notes.

\textbf{Example:} Consider
\begin{align*}
A=\begin{pmatrix}[c|c:c]
A_{1,1} & A_{1,2}&A_{1,3}\\\hline
A_{2,1} & A_{2,2}&A_{2,3}\\\hdashline
A_{3,1} & A_{3,2}&A_{3,3}
\end{pmatrix}
\end{align*}
and take the quasi-determinant with respect  the first diagonal block, which we define as the Schur complement indicated by the non dashed lines
\begin{align*}
\Theta_1(A)&=\begin{vmatrix}  A_{11,1} &\begin{matrix} A_{1,2} &A_{1,3}\end{matrix}\\\begin{matrix}
A_{2,1}\\A_{3,1}
\end{matrix} &\boxed{\begin{matrix}A_{2,2}&A_{2,3}
\\ A_{3,2}&  A_{3,3}\end{matrix}}
\end{vmatrix}={\begin{pmatrix}
A_{2,2}&A_{2,3}\\
A_{3,2}&A_{3,3}
\end{pmatrix}}-
{\begin{pmatrix}
A_{2,1} \\
A_{3,1}
\end{pmatrix}}A_{1,1}^{-1}{\begin{pmatrix}
A_{1,2}&A_{1,3}
\end{pmatrix}}
\\&=
\begin{pmatrix}[c:c]
A_{2,2}- A_{2,1}A_{1,1}^{-1}A_{1,2}  & A_{2,3}-A_{2,1}A_{1,1}^{-1}A_{1,3}\\\hdashline
A_{3,2}-A_{3,1}A_{1,1}^{-1} A_{1,2}& A_{3,3}-A_{3,1}A_{1,1}^{-1} A_{1,3}
\end{pmatrix}.
\end{align*}

Take the quasi-determinant given by the Schur complement as indicated by the dashed lines
\begin{align*}
\Theta_2(\Theta_1(A))&=\begin{vmatrix}
A_{2,2}- A_{2,1}A_{1,1}^{-1}A_{1,2}  & A_{2,3}-A_{2,1}A_{1,1}^{-1}A_{1,3}\\
A_{3,2}-A_{3,1}A_{1,1}^{-1} A_{1,2}& \boxed{A_{3,3}-A_{3,1}A_{1,1}^{-1} A_{1,3}}
\end{vmatrix}\\
&=
{{\begin{multlined}[t][0.75\textwidth]
A_{3,3}-A_{3,1}A_{1,1}^{-1} A_{1,3}-  \big(A_{3,2}-A_{3,1}A_{1,1}^{-1} A_{1,2}\big)\big(A_{2,2}- A_{2,1}A_{1,1}^{-1}A_{1,2}\big)^{-1}\big(A_{2,3}-A_{2,1}A_{1,1}^{-1}A_{1,3}\big)
\end{multlined}}}
\end{align*}
Compute, for the very same matrix
\begin{align*}
A={\begin{pmatrix}[cc|c]
A_{1,1} & A_{1,2}&A_{1,3}\\
A_{2,1} & A_{2,2}&A_{2,3}\\\hline
A_{3,1} & A_{3,2}&A_{3,3}
\end{pmatrix}}
\end{align*}
the Schur complement indicated by the non-dashed lines, that is,
\begin{align*}
\Theta_{\{1,2\}}(A)=&\begin{vmatrix}  A_{1,1} & A_{1,2} & A_{1,3}\\ A_{2,1} & A_{2,2} &A_{2,3}\\
A_{1,3}& A_{2,3} & \boxed{A_{3,3}}
\end{vmatrix}=A_{3,3}- {\begin{pmatrix}
A_{3,1} & A_{3,2}
\end{pmatrix}}
{\begin{pmatrix}
A_{1,1} & A_{1,2}\\
A_{2,1} & A_{2,2}
\end{pmatrix}}^{-1}
{\begin{pmatrix}
A_{1,3}\\
A_{2,3}
\end{pmatrix}}.
\end{align*}

Now, from
\begin{align*}
{\begin{pmatrix}
A_{1,1} & A_{1,2}\\
A_{2,1} & A_{2,2}
\end{pmatrix}}^{-1}= {\begin{pmatrix}
A_{1,1}^{-1}+A_{1,1}^{-1}A_{1,2}( A_{2,2}-A_{2,1}A_{1,1}^{-1}A_{1,2})^{-1}A_{2,1}A_{1,1}^{-1}
&-A_{1,1}^{-1}A_{1,2}( A_{2,2}-A_{2,1}A_{1,1}^{-1}A_{1,2})\\
-( A_{2,2}-A_{2,1}A_{1,1}^{-1}A_{1,2})^{-1}A_{2,1}A_{1,1}^{-1} &( A_{2,2}-A_{2,1}A_{1,1}^{-1}A_{1,2})^{-1}
\end{pmatrix}},
\end{align*}
we deduce   ${\Theta_2(\Theta_1(A))=\Theta_{\{1, 2\}}(A)}$, the simplest case of the heredity principle.

\begin{pro}[Heredity Principle]
Quasi-determinants of quasi-determinants are quasi-determinants.
\end{pro}

Given any set $I=\{i_1,\dots,i_m\}\subset\{1,\dots,k\}$ the heredity principle allows us to define the quasi-determinant
\begin{align*}
\Theta_I(A)=\Theta_{i_1}(\Theta_{i_2}(\cdots\Theta_{i_m}(A)\cdots)).
\end{align*}

\begin{pro}[Quasi-determinantal expressions]
	The sequence of biorthogonal polynomials an its squared norms are quasi-determinants
	 \begin{align*}
		P_{1,k}(x)&=\Theta_*\begin{bmatrix}[c c c c c|c]
		&         & G^{[k]} &  &           & \begin{matrix}  1 \\ x \\ \vdots\\ x^{k-1} \end{matrix}\\
		\hline
		G_{k,0} & G_{k,1} & \dots   &  & G_{k,k-1} & x^k \end{bmatrix},\\
		H_{k}&=\Theta_*(G^{[k+1]}),\\
	P_{2,k}(x)&=\Theta_*\begin{bmatrix}[c c c c c|c]
	&         & G^{[k]} &  &           & \begin{matrix}  G_{0,k} \\ G_{1,k} \\ \vdots\\ G_{k,k-1} \end{matrix}\\
	\hline
	1&x & \dots   &  & x^{k-1} & x^k \end{bmatrix}.
	\end{align*}
\end{pro}

\begin{proof}
From $LU$ factorization,
recalling that  $S_{k,k}=1$,  we get
\begin{align*}
\begin{pmatrix} S_{k,0} & S_{k,1} & \dots & S_{k,k-1} \end{pmatrix} G^{[k]}=
-\begin{pmatrix} G_{k,0} & G_{k,1} & \dots & G_{k,k-1} \end{pmatrix}.
\end{align*}
As  $G^{[k]}$ is a  non singular  matrix
\begin{align*}
\begin{pmatrix} S_{k,0} & S_{k,1} & \dots & S_{k,k-1} \end{pmatrix}=
-\begin{pmatrix} G_{k,0} & G_{k,1} & \dots & G_{k,k-1} \end{pmatrix} \left( G^{[k]} \right)^{-1}.
\end{align*}
On the other hand, 
{\begin{align*}
	H_k=\begin{pmatrix} S_{k,0} & S_{k,1} & \dots & S_{k,k-1}& S_{k,k}\end{pmatrix}
	\begin{pmatrix}
	G_{0,k} \\
	G_{1,k} \\
	\vdots \\
	G_{k-1,k} \\
	G_{k,k}
	\end{pmatrix}=
	\begin{pmatrix} S_{k,0} & S_{k,1} & \dots & S_{k,k-1}\end{pmatrix}
	\begin{pmatrix}
	G_{0,k} \\
	G_{1,k} \\
	\vdots \\
	G_{k-1,k} 
	\end{pmatrix}+G_{k,k}.
	\end{align*}}
\end{proof}

\subsection{Second kind functions and $LU$ factorizations}
\begin{defi}[Second kind functions]
	For  a  generalized kernel $u_{x,y}$ we define two families of second kind functions  given by
	\begin{align*}
	C_{1,n}(z)&=\left\langle P_{1,n}(x),\frac{1}{z-y}\right\rangle_u, &z&\not\in\operatorname{supp}_y(u), &
	C_{2,n}(z)&=\left\langle \frac{1}{z-x},P_{2,n}(y)\right\rangle_{u}, &z&\not\in\operatorname{supp}_x(u).
	\end{align*}
\end{defi}

\begin{pro}[$LU$ factorization representation of second kind functions]
For $z$ such that $|z|>\sup\big(|y|: y\in\operatorname{supp}_y(u)\big)$ it follows that
\begin{align*}
C_{1}(z)= H S_2^{-\top}\chi^*(z),
\end{align*}
while  for  $z$ such that $|z|>\sup\big(|x|: x\in\operatorname{supp}_x(u)\big)$ we find
\begin{align*}
C_{2}(z)&=H S_1^{-1}  \chi^*(z).
\end{align*}
\end{pro}
\begin{proof}
	Whenever $z$ belongs to an annulus around the origin with no intersection with the $y$ support  of the functional
\begin{align*}
C_{1}(z)&=S_1\left\langle \chi(x),\frac{1}{z-y}\right\rangle_u=S_1\left\langle \chi(x),\frac{1}{z}\sum_{k=0}^\infty \frac{y^k}{z^k}\right\rangle_u=
S_1 \left\langle \chi(x),(\chi(y))^\top \chi^*(z)\right\rangle_u=S_1G \chi^*(z)= H S_2^{-\top}\chi^*(z).
\end{align*}
When $z$ belongs to an annulus around the $0$ without intersection with the $x$ support  of the functional
\small\begin{align*}
(	C_{2}(z))^\top&=\left\langle \frac{1}{z-x}, (\chi(y))^\top\right\rangle_u (S_2)^\top=\left\langle \frac{1}{z}\sum_{k=0}^\infty \frac{x^k}{z^k}, (\chi(y))^\top\right\rangle_u(S_2)^\top=
\left\langle (\chi^*(z))^\top\chi(x), (\chi(y))^\top \right\rangle_u(S_2)^\top\\&=(\chi^*(z))^\top G  (S_2)^\top= (\chi^*(z))^\top S_1^{-\top} H.
\end{align*}
\end{proof}

\subsection{Spectral matrices}
Let us introduce the shift or spectral matrix
\begin{align*}
\Lambda&:=\begin{pmatrix}
0 & 1 & 0 & 0 &\dots \\
0 & 0 & 1 & 0 &\dots \\
0 & 0 & 0 & 1 &\dots\\
\vdots & \vdots & \vdots &\ddots 
\end{pmatrix}.
\end{align*}

The following spectral properties hold
\begin{align*}
\Lambda \chi(x)&=x\chi(x), &
\Lambda \chi^*(x)&=\frac{1}{x}\chi^*(x).
\end{align*}

If $(E_{i,j})_{s,t}:=\delta_{s,i}\delta_{t,j}$ we have
\begin{itemize}
	\item In the one hand  $\Lambda \Lambda^{\top}=\mathbb{I}$; and, in the other hand, $\Lambda^{\top} \Lambda= \mathbb{I}-E_{0,0}$.
	\item $\Lambda^{\top} \chi(x)=\frac{1}{x}(\mathbb{I}-E_{0,0})\chi(x)$.
\end{itemize}

	We introduce the semi-infinite matrices
	\begin{align*}
	J_1&:=S_1\Lambda (S_1)^{-1}, & J_2&:=S_2\Lambda (S_2)^{-1}.
	\end{align*}

\begin{pro}[Spectral matrices are Hessenberg]
The spectral matrices are lower uni-Hessenberg matrices, i.e., of the form
\begin{align*}
\begin{pmatrix}
* & 1 &0&0&\dots\\
*&*&1&0&\dots\\
\vdots &\ddots&\ddots&\ddots&\ddots
\end{pmatrix}.
\end{align*}
\end{pro}
\begin{pro}[Spectrality]
The spectral matrices satisfy the eigenvalue property
\begin{align*}
J_1 P_1(z)&=z P_1(z), & J_2P_2(z)&=zP_2(z).
\end{align*}
\end{pro}

\begin{pro}[Eigenvalues of the truncation and roots of the polynomials]
	The roots of  $P_{i,k}(z)$ and the eigenvalues of the truncation $J_i^{[k]}$ coincide.
\end{pro}

\begin{proof}
	 We have \begin{align*}
J_i^{[k]}\begin{pmatrix}
P_{i,0}(x) \\
P_{i,1}(x) \\
\vdots \\
P_{i,k-1}(x)
\end{pmatrix}=
x\begin{pmatrix}
P_{i,0}(x) \\
P_{i,1}(x) \\
\vdots \\
P_{i,k-1}(x)
\end{pmatrix}-
\begin{pmatrix}
0 \\
0 \\
\vdots \\
P_{i,k}(x)
\end{pmatrix}.
\end{align*}
For a root $\alpha$, i.e., $P_{i,k}(\alpha)=0$ we arrive to
\begin{align*}
\left(J^{[k]}\right) \begin{pmatrix}
P_{i,0}(\alpha) \\
P_{1}(\alpha) \\
\vdots \\
P_{i,k-1}(\alpha)
\end{pmatrix}=
\alpha \begin{pmatrix}
P_{i,0}(\alpha) \\
P_{i,1}(\alpha) \\
\vdots \\
P_{i,k-1}(\alpha)
\end{pmatrix}
\end{align*}
and, therefore, we have the eigenvector $ \begin{psmallmatrix}
P_{i,0}(\alpha) \\
P_{i,1}(\alpha) \\
\vdots \\
P_{i,k-1}(\alpha)
\end{psmallmatrix}$ with eigenvalue $\alpha$.
\end{proof}

\subsection{Christoffel--Darboux kernels}

\begin{defi}[Christoffel--Darboux kernels]
	Given two sequences of matrix  biorthogonal polynomials 
	\begin{align*}
	\text{ $\big\{P_{1,k}(x)\big\}_{k=0}^\infty$ and $\big\{P_{2,k}(y)\big\}_{k=0}^\infty$, }
	\end{align*}with respect to the sesquilinear form $\prodint{\cdot,\cdot}_u$, we define the $n$-th Christoffel--Darboux kernel matrix polynomial
	\begin{align*}
	K_{n}(x,y):=\sum_{k=0}^{n}P_{2,k}(y)( H_k)^{-1}P_{1,k}(x),
	\end{align*}
	and the  mixed Christoffel--Darboux kernel
	\begin{align*}
	K^{\operatorname{mix}}_n(x,y)&:=\sum_{k=0}^nP_{2,k}(y)(H_k)^{-1}C_{1,k}(x).
	\end{align*}
\end{defi}

\begin{pro}[Projection properties]
\begin{enumerate}
	\item 	For a quasidefinite generalized kernel $u_{x,y}$, 
	the corresponding  Christoffel--Darboux kernel gives the projection operator%
	\begin{align*}
	\prodint{ K_n(x,z),\sum_{j=0}^m \lambda_j P_{2,j}(y)}_u&=
	\sum_{j=0}^n\lambda_j P_{2,j}(z),&	\prodint{ \sum_{j=0}^m\lambda_jP_{1,j}(x),K_n(z,y)}_u&=
	\sum_{j=0}^nC_jP_{1,j}(z).
	\end{align*}
	for any $m\in\{0,1,2,\dots\}$.
	\item
	In particular, we have
	\begin{align*}
	\prodint{ K_n(x,z),y^l}_u&=z^l, & l\in&\{0,1,\dots,n\}.
	\end{align*}
\end{enumerate}
\end{pro}

\begin{pro}[ABC Theorem (Aitken, Berg and Collar) \cite{simon}]
	We have the following relation
\begin{align*}
K^{[l]}(x,y)=\left(\chi^{[l]}(y)\right)^{\top}\left(G^{[l]}\right)^{-1}\chi^{[l]}(x).
\end{align*}
\end{pro}

\begin{proof}
Is a consequence of the following	
		\begin{align*}
\left(\chi^{[l]}(y)\right)^{\top}\left(G^{[l]}\right)^{-1}\chi^{[l]}(x)&=
\left(\chi^{[l]}(y)\right)^{\top}
\left((S_1^{[l]})^{-1}H^{[l]}(S_2^{[l]})^{-\top}\right)^{-1}\chi^{[l]}(x)=\left(S_2^{[l]}\chi^{[l]}(y)\right)^{\top}\left(H^{[l]}\right)^{-1}S_1^{[l]}\chi^{[l]}(x)\\
&=\left(P_2^{[l]}(y)\right)^{\top}\left(H^{[l]}\right)^{-1}P_1^{[l]}(x).
\end{align*}
\end{proof}

\begin{pro}[Reproducing  property]
As we are dealing with a projection we find
\begin{align*}
\langle K^{[l]}(x,z_2),K^{[l]}(z_1,y)\rangle_{u}=K^{[l]}(z_1,z_2).
\end{align*}
\end{pro}

\begin{proof}
	As an	exercise, let us use the ABC theorem
\begin{multline*}
\langle K^{[l]}(x,z_2),K^{[l]}(z_1,y)\rangle_{u}
\begin{aligned}[t]&=
\left\langle
\left(\chi^{[l]}(z_2)\right)^{\top}
\left(G^{[l]}\right)^{-1}
\chi^{[l]}(x),
\left(\chi^{[l]}(y)\right)^{\top}
\left(G^{[l]}\right)^{-1}
\chi^{[l]}(z_1)
\right\rangle_u
\\&=\left(\chi^{[l]}(z_2)\right)^{\top}\left(G^{[l]}\right)^{-1}\chi^{[l]}(z_1).
\end{aligned}
\end{multline*}	
\end{proof}

\section{Standard orthogonality: Hankel reduction}

Recall that for bilinear forms associated to a Borel measure or a linear functional the Gram matrix is a Hankel matrix, $G_{i,j+1}=G_{i+1,j}$. We will consider in this section some properties that appear in this situation  and not in the general scheme.

\subsection{Recursion relations}
This property is just a reflection of the self-adjointness of the multiplication operator by  $x$  with respect to inner product is reflected in the Hankel structure of the moment matrix
\begin{align*}
\langle xf,h \rangle_{\mu}=\langle f,xh \rangle_{\mu} \Rightarrow \Lambda G=G\Lambda^{\top}\Rightarrow G_{i,j}=G_{i+j}.
\end{align*}

That leads to the tridiagonal form of the spectral matrices, now named after Jacobi.
\begin{pro}
	The spectral matrices are linked  by
$	J_1H=H(J_2)^\top$.\end{pro}
\begin{proof}
	From the $LU$ factorization (Cholesky) $G=S^{-1} HS^{-\top}$ and the symmetry $\Lambda G=G\Lambda^\top$ we find
	\begin{align*}
	S\Lambda S^{-1}=H (S\Lambda S^{-1})^{\top} H^{-1}.
	\end{align*}
\end{proof}
This tridiagonal Jacobi matrix $J:=J_1$ can be written as follows
\begin{align*}
J&:= \begin{pmatrix}
-S_{10}             & 1                                                        & 0                                                                                         & 0                                                           & 0                                     & 0     &\dots\\
H_1 H_0^{-1}        & S_{10}-S_{21}                   & 1                                                                                         & 0                                                           & 0                                     & 0     &\dots\\
0                           & H_2 H_1^{-1}                              & S_{21}-S_{32}                                                  & 1                                                           & 0                                     & 0     &\dots\\
0                           & 0                                                        & H_3 H_2^{-1}                                                               &S_{32}-S_{43}                     & 1                                     & 0     &\dots\\
0                           & 0                                                        & 0                                                                                         & H_4 H_3^{-1}                                 & S_{43}-S_{54} & 1     &\dots\\
0                           & 0                                                        & 0                                                                                         & 0                                                           & \ddots                                &\ddots &\ddots\\
\vdots                      & \vdots                                                   & \vdots                                                                                    &\vdots 
\end{pmatrix}.
\end{align*}
Therefore, the spectral properties lead to the well known recursion relations.

\begin{pro}[3-term relations]
	The orthogonal polynomials and the corresponding second kind functions fulfill
		\begin{align*} 
	JP(x)&= x P(x), & \frac{H_k}{H_{k-1}} P_{k-1}+(S_{k,k-1}-S_{k+1,k})P_{k}+P_{k+1}=xP_k,\\
	JC(x)&=xC(x)-H_{0}\boldsymbol{e}_0, & 
	\frac{H_k}{H_{k-1}} C_{k-1}+(S_{k,k-1}-S_{k+1,k})C_{k}+C_{k+1}=xC_k,
	\end{align*}
\end{pro}
where $\boldsymbol{e}_0:=\left(1,0,0,\dots,0,\dots \right)^{\top}$

\subsection{Heine formulas}

As the Gram matrix now is a moment matrix we find the well known Heine formulas:
\begin{pro}[Heine integral representation]
	  The orthogonal polynomials can be written as follows
	  \begin{align*}	
		P_{k}(x)&=\frac{1}{k! \det\left[G^{[k]}\right]}\int \prod_{j=1}^k (x-x_j)\prod_{1\leq j<n\leq k} (x_n-x_j)^2 \d \mu(x_1) \d \mu(x_2) \dots \d \mu(x_k)
		\end{align*}
\end{pro}
\begin{proof}
	 From the quasi-determinantal expression we get 
\begin{align*}
P_{k}(x)&=\frac{1}{\det\left[G^{[k]}\right]} \det \begin{bmatrix}[c c c c c|c]
&         & G^{[k]} &  &           & \begin{matrix}  1 \\ x \\ \vdots\\ x^{k-1} \end{matrix}\\
\hline
G_{k,0} & G_{k,1} & \dots   &  & G_{k,k-1} & x^k \end{bmatrix},
\end{align*}
and using the  Vandermonde formula we get the result.
\end{proof}

\subsection{Gauss quadrature formula}

\begin{pro}[Eigenvalues of the truncation of the spectral matrix]
	Assume that  the measure  $\mu$ does not change sign in its support $(a,b)$. Then, all eigenvalues of the truncations of the spectral matrices $J_1$ and $J_2$  belong to  $(a,b)$ and are simple.
\end{pro}
\begin{proof}
	Let   $\{a_i\}_{i=1}^m\subset(a,b)$ be the points  where the polynomial $P_n$ changes sign in $(a,b)$. Then, as $P_n(z)$ has  $n$ roots, $m\leq n$. Therefore, $(x-a_1)\dots(x-a_m)P_n(x)$ does not change sign in $(a,b)$, but from the orthogonality relations we know that $P_n$ is orthogonal to any polynomial of degree less than $n$ and, consequently,
	\begin{align*}
	\int_{a}^b(x-a_1)\dots(x-a_m)P_n(x)\d\mu(x)=0,
	\end{align*}
	for $m<n$, which is, as $\mu$ does not change sign in its support,  contradictory. Hence, the only possibility is to have $m=n$.
\end{proof}

\begin{pro}[Powers of truncations of the spectral  matrices]
	\begin{align*}
	\left(J^{[k]}\right)^{j}&=\mathbb{P}\begin{pmatrix}
	\alpha_{k,1}^j         &                     &          &             \\
	& \alpha_{k,2}^j        &          &             \\
	&                     & \ddots   &             \\
	&                     &          & \alpha_{k,k}^j
	\end{pmatrix} \mathbb{P}^{-1}, &
	\mathbb{P}&:=\begin{pmatrix}
	1                   & 1                   & \dots & 1                   \\
	P_{1}(\alpha_{k,1}) & P_{1}(\alpha_{k,2}) & \dots & P_{1}(\alpha_{k,k})  \\
	\vdots &&&\vdots\\
	P_{k-1}(\alpha_{k,1}) & P_{k-1}(\alpha_{k,2}) & \dots & P_{k-1}(\alpha_{k,k}) 
	\end{pmatrix}.
	\end{align*}
\end{pro}

\begin{pro}The following identity holds
\begin{align*}
\int_{\Omega} x^{j} \d \mu(x)=(J^j)_{0,0} H_0.
\end{align*}
\end{pro}
\begin{proof}It follows from 3-term relation
\begin{align*}
\begin{aligned}
x^jP(x)&=J^jP(x)    &    &\Longrightarrow  &   x^j&=(J^j)_{0,0}+(J^j)_{0,1}P_1(x)+\dots+(J^j)_{0,j}P_j(x),
\end{aligned}
\end{align*}
and, consequently,
\begin{align*}
\int_{\Omega} x^{j} \d \mu(x)= (J^j)_{0,0}\langle 1,1 \rangle+
(J^j)_{0,1}\langle P_1,1\rangle+\dots+(J^j)_{0,j}\langle P_j,1\rangle= (J^j)_{0,0} H_0.
\end{align*}
\end{proof}

\begin{pro}[Gaussian cuadrature]
The roots $\{\alpha_{k,l}\}_{l=1}^k$ of the orthogonal polynomials  $P_k(x)$ are the  $k$ points for the quadrature  of  $\mu$ with
precision  $ 2k-1$. Namely,
\begin{align*}
\int_{\Omega} x^{j} \d \mu(x)&= \sum_{l=1}^{k} \lambda_{j,l} \alpha_{k,l}^{j},& 0&\leq j \leq 2k-1,
\end{align*}
for some coefficients $\lambda_{j,l}$.
\end{pro}

\begin{proof}
	From the  $(2+1)$-diagonal structure of the Jacobi matrix $J$ it follows that
$\left(J^j\right)_{0,0}=\left(\left[J^{[k]}\right]^j \right)_{0,0}$ whenever $0\leq j \leq 2k-1$.

Now, using the diagonal form of the powers of $J$ and denoting  by
$\mathbb{P}^{-1}\rfloor_{1}$ the first column of  $\mathbb{P}^{-1}$
\begin{align*}
\int_{\Omega} x^{j} \d \mu(x)= \left(\left[J^{[k]}\right]^j \right)_{0,0}H_0=
H_0 \begin{pmatrix} \alpha_{k,1}^j & \alpha_{k,2}^j & \dots & \alpha_{k,k}^j \end{pmatrix} \mathbb{P}^{-1}\rfloor_1=
\sum_{l=1}^{k} \lambda_{j,l} \alpha_{k,l}^{j},  
\end{align*}
for $0\leq j \leq 2k-1$.

\end{proof}

\subsection{Christoffel--Darboux formula}
\begin{pro}[Christoffel--Darboux formula]
	The Christoffel--Darboux kernel satisfies
	\begin{align*}
	(x- y)K_n(x,y)&=P_n(y) (H_n)^{-1}P_{n+1}(x)-P_{n+1}(y)(H_{n})^{-1}P_{n}(x),
	\end{align*}
	with confluent version given by
	\begin{align*}
	\sum_{k=0}^{l-1}\frac{(P_k(x))^2}{H_k}=\frac{1}{H_{l-1}}\Big(P'_{l}(x)P_{l-1}(x)-P'_{l-1}(x)P_{l}(x)\Big)
	\end{align*}
\end{pro}

\begin{proof}
	From the eigenvalue property we get 
\begin{align*}
\left( H^{-1}J\right)^{[l]} P(y)^{[l]}+\left( H^{-1}J\right)^{[l,\geq l]} P(y)^{[\geq l]}&=y \left( H^{-1}\right)^{[l]} P(y)^{[l]}, \\
[P(x)^{\top}]^{[l]} \left( H^{-1}J\right)^{[l]}+[P(x)^{\top}]^{[\geq l]} \left( H^{-1}J\right)^{[\geq l,l]}&=x[P(x)^{\top}]^{[l]}\left( H^{-1}\right)^{[l]}.
\end{align*}
Left multiply the first eq. by  $[P(x)^{\top}]^{[l]}$ and right multiply the second by $P(y)^{[l]}$,  and subtract both results to get
\begin{align*}
[P(x)^{\top}]^{[l]}\left( H^{-1}J\right)^{[l,\geq l]}P(y)^{[\geq l]}-
[P(x)^{\top}]^{[\geq l]} \left( H^{-1}J\right)^{[\geq l,l]}P(y)^{[l]}=
(y-x)[P(x)^{\top}]^{[l]}\cdot \left( H^{-1}\right)^{[l]} P(y)^{[l]}
\end{align*}
and the result follows. To prove the confluent case just take the limit  $y\rightarrow x$  in the previous case. 
\end{proof}

\begin{pro}[Christoffel--Darboux formula for mixed kernels]
The mixed Christoffel-Darboux kernel fulfills
\begin{align*}
(x-{y})	K^{\operatorname{(mix})}_{n}(x,y)
=P_{n}(y)(H_n)^{-1}C_{n+1}(x)-P_{n+1}(y)(H_n)^{-1}C_{n}(x)+1
\end{align*}
\end{pro}

\begin{proof}
	As in the no mixed case consider the expression $P(x)^{\top}H^{-1}J C(y)$ and the two possible ways of computing them 
either as 
$P(x)^{\top}\left[H^{-1}J C(y)\right]$ or as $\left[P(x)^{\top}H^{-1}J\right] C(y)$.

\end{proof}

\section{Very classical orthogonal polynomials: Hermite, Laguerre and Jacobi polynomials}

Here we study the definite positive case. Neither Bessel polynomials nor Laguerre and Jacobi polynomials with right parameters leading to quasi definite linear functionals are considered.
The very classical orthogonal polynomials, Hermite, Laguerre and Jacobi can be understood \emph{à la Bochner } as the eigenfunctions of second order differential operators, associated with corresponding definite positive Borel measures. They also can be characterized by a Rodrigues formula or if you want as those that when derivated preserve the orthogonal character.

\begin{defi}[Pearson equation]
	The weight  $u_{\gamma}$ is said very classical whenever we have polynomials
	$p_2(x)=a x^2+bx+c$ and $p_{1,\gamma}(x)=(A_{\gamma}-2a)x+(B_{\gamma}-b)$ with $A_{\gamma}\neq 0$ such that $u_{\gamma}$ 
	satisfies
	\begin{align*}
	p_2(x) \frac{\d}{\d x} u_{\gamma} =p_{1,\gamma}(x) u_{\gamma}
	\end{align*}
\end{defi}

The well known very classical weights are:
\begin{itemize}
	\item Hermite $u(x)=\operatorname{e}^{-x^2}$, $x\in \mathbb{R}$ with $p_{1}=-2x$, $p_2=1$
	\item Laguerre $u_{\alpha}(x)=x^{\alpha} \operatorname{e}^{-x}$, $\alpha >-1$, $x\in \mathbb{R}_+$, with  $p_{1,\alpha}=(\alpha-x)$, $p_2=x$
	\item Jacobi $u_{\alpha,\beta}(x)=(1-x)^{\alpha}(1+x)^{\beta}$, $\alpha,\beta > -1$, $x\in(-1,1)$, with $p_{1,\alpha,\beta}=-[(\alpha-\beta)+(\alpha+\beta)x]$, $p_2=1-x^{2}$
\end{itemize}
They depend upon: zero ($\gamma=\{\varnothing\}$), one ($\gamma=\{\alpha\}$) and  two parameters ($\gamma=\{\alpha, \beta\}$) respectively. 

\begin{itemize}
\item We denote $u_{\gamma+1}(x)$ the action of increasing by one all the parameters in a classical weight $u_\gamma(x)$. For example,  Hermite do not change (no parameter present), in Laguerre we have $\alpha\mapsto\alpha+1$ while in Jacobi the shift  is  $(\alpha, \beta)\mapsto(\alpha+1, \beta+1)$
\item	$p_2(x) u_{\gamma}(x)= u_{\gamma+1}(x)$
\end{itemize}

\begin{pro}[Properties]
	For the very classical weights we have
 \begin{align*}
p_2 u_{\gamma}\big{\lvert}_{\partial \Omega }&=u_{\gamma+1}\big{\lvert}_{\partial \Omega }=0, &
\langle p_2 f',h \rangle_{u_\gamma}&= -\langle f,(p_2'+p_1)h\rangle_{u_\gamma}-  \langle f,p_2 h' \rangle_{u_\gamma}.
\end{align*}
\end{pro}

We can matrix represent the action of derivation with the semi-infinite matrix
	\begin{align*}	D&:=\begin{pmatrix}
0 & 0 & 0 & 0 &\dots \\
1 & 0 & 0 & 0 &\dots \\
0 & 2 & 0 & 0 &\dots\\
0 & 0 & 3 & 0 &\dots\\
0 & 0 & 0 & 4 &\ddots\\
\vdots & \vdots & \vdots &\vdots 
\end{pmatrix}, &
D \chi(x)=\chi(x)'.
\end{align*}
Using $f(x)=\left(f_0,f_1,\dots \right)\chi(x)=\boldsymbol{f}^{\top}\chi(x)$ 
we can write
 	\begin{align*}
xf(x)&=\boldsymbol{f}^{\top} \Lambda \chi(x),  & f(x)'&=\boldsymbol{f}^{\top} D \chi(x).
\end{align*}

\begin{pro}[Symmetry of the moment matrix]
	The  Gram  matrices of the classical weights are linked by
\begin{align*}
D G_{\gamma+1}= -G_{\gamma} \Big(p_2'(\Lambda)+p_{1,\gamma}(\Lambda)+D p_2(\Lambda)\Big)^{\top}.
\end{align*}
\end{pro}

\begin{pro}[$LU$ factorization and classical weights]
The $LU$ factorization  of the Gram matrix of the classical weights leads to
\begin{align*}
S_{\gamma} D S_{\gamma+1}^{-1}&= -H_{\gamma}\Big(S_{\gamma+1}\big(p_2'(\Lambda)+p_{1,\gamma}(\Lambda)+Dp_2(\Lambda) \big)S_{\gamma}^{-1} \Big)^\top H_{\gamma+1}^{-1}.
\end{align*}
Equivalently,
\begin{align*}
(S_{\gamma})_{n+1,n}&=(n+1)\frac{B+nb}{A+2na}, &
(H_{\gamma})_n&=\frac{-n}{A_{\gamma}+(n-1)a}(H_{\gamma+1})_{n-1}.
\end{align*}
\end{pro}

\begin{pro}[Self-adjoint differential operator]
	For the classical weights we have a self-adjoint second order differential operator, i.e.,
	\begin{align*}
\left\langle \left[ p_2 \frac{\d^2}{\d x^2}+(p_2'+p_{1,_{\gamma}}) \frac{\d}{\d x}\right] f, h\right\rangle_{u_\gamma}&=
\left\langle f,\left[ p_2 \frac{\d^2}{\d x^2}+(p_2'+p_{1,_{\gamma}}) \frac{\d}{\d x}\right] h\right\rangle_{u_\gamma}.
\end{align*}
\end{pro}

\begin{pro}[Semi-infinite matrix version]
	The matrices of classical moments enjoy the following additional symmetry given by the matrix representation of
	a linear second-order differential operator with polynomial coefficients
		\begin{align*}
	[D^2(a\Lambda^2+b\Lambda+c)+D(A_{\gamma}\Lambda+B_{\gamma})] G_{\gamma}= G_{\gamma} \left[ D^2(a\Lambda^2+b\Lambda+c)+D(A_{\gamma}\Lambda+B_{\gamma}) \right]^\top.
	\end{align*}
\end{pro}

Observe that calling $M:=S D S^{-1}$, which is a  strictly lower triangular matrix  with first subdiagonal the sequence of natural numbers  fulfilling $[J,M]=I$, the above relation leads to
\begin{align*}
	[M^2(a J^2+bJ+c)+M(A_{\gamma}J+B_{\gamma})] = \left[ M^2(aJ^2+bJ+c)+M(A_{\gamma}J+B_{\gamma}) \right]^\top.
\end{align*}

\begin{pro}[Diagonalizing the self-adjoint differential operator]
\begin{align*}
N_{\gamma}:=S_{\gamma} \left[ D^2(a\Lambda^2+b\Lambda+c)+D(A_{\gamma}\Lambda+B_{\gamma})\right] S_{\gamma}^{-1}= \begin{pmatrix}
0             & 0                   & 0                    & 0               & 0                &\dots\\
0             & A_{\gamma}          & 0                    & 0               & 0                &\dots\\
0             & 0                   & 2(A_{\gamma}+a)      & 0               & 0                &\dots\\
0             & 0                   & 0                    &3(A_{\gamma}+2a) & 0                &\dots\\
0             & 0                   & 0                    & 0               &       \ddots     &      \\
\vdots        & \vdots              & \vdots               &\vdots 
\end{pmatrix}.
\end{align*}
The diagonal  coefficients, $(N_{\gamma})_n=n(A_\gamma+(n-1)a)$, are the eigenvalues of the sequence of classical orthogonal polynomials, being these one the corresponding eigenfunctions
\begin{align*}
F_{\gamma}&:=p_2 \frac{\d^2}{\d x^2}+(p_2'+p_{1,_{\gamma}}) \frac{\d}{\d x}  &     &\Longrightarrow &
F_{\gamma} \left[P_{\gamma}(x)\right]&=N_{\gamma} P_{\gamma}(x).
\end{align*}
In particular, for the three classical families, we have
\begin{itemize}
\item \textbf{Hermite}. $\dfrac{\d u}{\d x}=-2x u$; 
$A=-2$, $B=0$, $a=0$, $b=0$, $c=1$;  $H_0=\sqrt{\pi}$; 
\begin{align*}
S_{n+1,n}&=0, &  
H_n&=\sqrt{\pi}\frac{n!}{2^n}, & 
N_n&=-n.
\end{align*}
\item \textbf{Laguerre}. $x\dfrac{\d u_{\alpha}}{\d x}=(\alpha-x) u_{\alpha}$; 
$A_{\alpha}=-1,\,\,B_{\alpha}=(1+\alpha)$, $a=0$, $b=1$, $c=0$; 
$(h_{\alpha})_0=\Gamma(\alpha+1)$; 
\begin{align*}
(S_{\alpha})_{n+1,n}&=-(n+1)[(n+1)+\alpha], &
(H_{\alpha})_n&=n!\Gamma(\alpha+n+1), &
(N_\alpha)_n&=-2n.
\end{align*}
\item \textbf{Jacobi}. 
$(1-x^2)\dfrac{\d u_{\alpha,\beta}}{\d x}=-[(\alpha-\beta)+(\alpha+\beta)x] u_{\alpha,\beta}$; 
$A_{\alpha,\beta}=-[(\beta+\alpha)+2]$, $B_{\alpha,\beta}=-(\alpha-\beta)$, $a=-1$, $b=0$, $c=1$;
$(H_{\alpha,\beta})_0=\dfrac{\Gamma(\alpha+1)\Gamma(\beta+1)2^{\alpha+\beta+1}}{(\alpha+\beta+1) \Gamma(\alpha+\beta+1)}$;
\begin{align*}
(S_{\alpha,\beta})_{n+1,n}&=\frac{(n+1)(\alpha-\beta)}{(\alpha+\beta+2)+2n}\\
(H_{\alpha,\beta})_n&=n!2^{(\alpha+\beta+2n+1)} \frac{\Gamma(\alpha+\beta+n+1)\Gamma(\alpha+n+1)\Gamma(\beta+n+1)}{(\alpha+\beta+2n+1)\Gamma^{2}(\alpha+\beta+2n+1)}\\
(N_{\alpha,\beta})_n&=-n(\beta-\alpha+1+n).
\end{align*}
\end{itemize}
\end{pro}

\section{Christoffel and Geronimus transformations}

\subsection{Some history}

Three perturbations have attracted the interest of the researchers.
\begin{wrapfigure}{l}{0.25\textwidth}
	\includegraphics[width=0.25\textwidth]{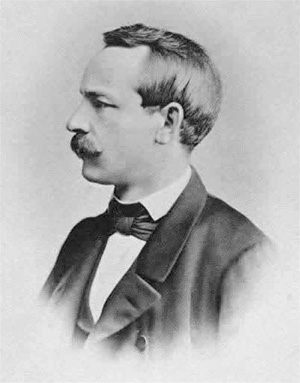}
\end{wrapfigure}
 Christoffel perturbations, that  appear when you consider a new functional $\hat{u}= p(x) u$, where $p(x)$ is a polynomial, 
were \textbf{studied in 1858 by the  German mathematician E. B.  Christoffel}  in \cite{Christoffel}
in the framework of Gaussian quadrature rules.

Christoffel  found  explicit formulas relating the corresponding sequences of orthogonal polynomials with respect to two measures,  the Lebesgue measure  $\d\mu$ supported in the interval $(-1,1)$ and  $d\hat{\mu}(x)= p(x) d\mu(x)$, with $p(x)=(x-q_1)\cdots(x-q_N)$ a signed polynomial in the support of $\d\mu$,  as well as the distribution of their zeros as nodes in such quadrature rules. Nowadays, these are  called\textbf{ Christoffel formulas}, and can be considered  a classical result in the theory of orthogonal polynomials which can be found in a number of  textbooks, see for example 
Chihara, Szeg\H{o} or Gautschi.

In the theory of orthogonal polynomials,  \textbf{connection formulas} between two families of orthogonal polynomials allow to express any polynomial of a given degree $n$ as a linear combination of all polynomials of degree less than or equal to $n$ in the second family. A noteworthy  fact regarding  the Christoffel finding is that in  this case the number of terms does not grow with the degree $n$ but remarkably,  and on the contrary,  remain constant,  equal to the degree of the perturbing polynomial.

\begin{wrapfigure}{r}{0.2\textwidth}
	\includegraphics[width=0.2\textwidth]{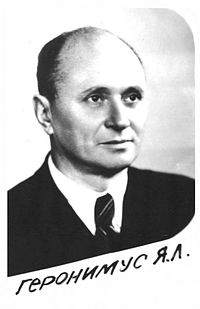}
\end{wrapfigure}Geronimus transformation appears when you are dealing with perturbed functionals $v$ defined by $p(x) v=u,$ where $p(x)$ is a polynomial. Such a kind of transformations were used by the \textbf{Russian mathematician J. L. Geronimus} in
\cite{Geronimus}
in order to have a nice proof of a result by W. Hahn  concerning the characterization of classical orthogonal polynomials (Hermite, Laguerre, Jacobi, and Bessel) as those orthogonal polynomials whose first derivatives are also orthogonal polynomials,  for an English account of Geronimus' paper  see \cite{Golinskii}.

Again, as happened  for the Christoffel transformation, within the Geronimus transformation one can find Christoffel type formulas, now in terms of the second kind functions, relating the corresponding sequences of orthogonal polynomials,   for example the work of P. Maroni  studied this situation for a  perturbation of the type $p(x)=x-a$.

Despite that in the paper by Geronimus no Christoffel type formula was derived,  in order to distinguish these Christoffel type formulas from those for Christoffel transformations,  we refer to them as Christoffel--Geronimus. formulas.

The more general problem related to linear functionals $u$ and $v$ satisfying $p(x)u= q(x)v,$ where $p(x), q(x)$ are polynomials has been analyzed the Russian mathematician V. B. Uvarov back in 1969   \cite{Uvarov1969connection} see also \cite{Zhedanov1997Rational} where the term linear spectral was given.

Uvarov found Christoffel type formulas,  see \cite{Uvarov1969connection},that allow for any pair of perturbing polynomials $p(x) $ and $q(x)$, to find the new orthogonal polynomials in terms of determinantal expressions of the original unperturbed second kind functions and orthogonal polynomials. On the other hand, the addition of a finite number of Dirac masses to a linear functional appears in the framework of the spectral analysis of fourth order linear differential operators with polynomial coefficients and with orthogonal polynomials as eigenfunctions. Therein you have the so called Laguerre-type, Legendre-type and Jacobi-type orthogonal polynomials introduced by H. L. Krall.

A more general analysis from the point of view of the algebraic properties of the sequences of orthogonal polynomials associated to the linear functionals $u$ and $w= u + \sum_{n=0}^{N} M_{n} \delta(x- a_{n})$, the so-called general Uvarov transformation by Zhedanov, see \cite{Zhedanov1997Rational}, has been done for the positive definite case by Uvarov.

\subsection{Christoffel and Geronimus transformations}
\begin{defi}[Christoffel transformations and the Schwartz kernel]
Given a   polynomial $W_C(x)$ of degree $N_C$,  and a generalized kernel $u_{x,y}\in \mathcal O_c'$  a  linear spectral or Geronimus--Uvarov transformation $\hat u_{x,y}$ of  $ u_{x,y}$ is a matrix of generalized kernels such that
\begin{align*}
\hat u_{x,y} =W_C(x)u_{x,y}.
\end{align*}
\end{defi}

\begin{pro}[Christoffel transformation and the bilinear form and Gram matrix]
	The perturbed Gram  matrix $\hat G:=\langle \chi(x),\chi(y)\rangle_{\hat u}$ and the original one $G$ satisfy
	\begin{align*}
	\hat  G =W_C(\Lambda)G.
	\end{align*}
	The sequilinear forms are related by
	\begin{align*}
	\prodint{P(x),  Q(y)W_G(y)}_{\hat u}=\prodint{ P(x)W_C(x), Q(y)}_u.
	\end{align*}
\end{pro}

\begin{defi}[Geronimus transformations and the Schwartz  kernel]
Given a  generalized kernel $u_{x,y}\in\mathcal O_c'$ with a given support $\operatorname{supp} u_{x,y}$, and a  polynomial $W(y)\in\mathbb C[y]$ of degree $N$,  such that $ \sigma(W(y))\cap \operatorname{supp}_y(u)=\varnothing$, a bivariate  generalized function $\check u_{x,y}$ is said to be a  Geronimus transformation of the  generalized kernels $u_{x,y}$ if
\begin{align*}
\check u_{x,y}W(y)=u_{x,y}.
\end{align*}
\end{defi}

\begin{pro}[Geronimus transformation and the bilinear form and Gram matrix]
In terms of  bilinear forms  a Geronimus transformation  fulfills
\begin{align*}
\prodint{P(x), Q(y)W(y)}_{\check u}=\prodint{P(x), Q(y)}_{u},
\end{align*}
while, in terms of the corresponding Gram matrices, satisfies
\begin{align*}
\check G W(\Lambda^\top)=G.
\end{align*}
\end{pro}

\begin{defi}[Linear spectral transformations and the Schwartz kernel]
Given two   polynomials $W_C(x),W_G(y)$ of degrees $N_C,N_G$,  and a generalized kernel $u_{x,y}\in \mathcal O_c'$ such that $\sigma(W_G(y)))\cap \operatorname{supp}_y(u)=\varnothing$, a matrix linear spectral or Geronimus--Uvarov transformation $\hat u_{x,y}$ of  $ u_{x,y}$ is a matrix of generalized kernels such that
\begin{align}
\hat u_{x,y} W_G(y)=W_C(x)u_{x,y}.
\end{align}
\end{defi}

\begin{pro}[Linear spectral transformation and the bilinear form and Gram matrix]
The perturbed Gram  matrix $\hat G:=\langle \chi(x),\chi(y)\rangle_{\hat u}$ and the original one $G$ satisfy
\begin{align*}
\hat  G W_G(\Lambda^\top)=W_C(\Lambda)G.
\end{align*}
The sequilinear forms are related by
\begin{align*}
\prodint{P(x),  Q(y)W_G(y)}_{\hat u}=\prodint{ P(x)W_C(x), Q(y)}_u.
\end{align*}
\end{pro}

\begin{pro}[Hankel case: Christoffel and Geronimus]
	Given monic  polynomials  $W_C(x):=\prod_{i=1}^{d}(x-r_i)^{m_i}$ and $W_G(x):=\prod_{i=1}^{s}(x-q_i)^{n_i}$ with degrees 
	$\sum_{i=1}^{d}m_i=M$ y $\sum_{i=1}^{s}n_i=N$, respectively, such that  $\{q_1,q_2,\dots q_s\}\cap \Omega=\varnothing$;
	the Christoffel  $(W_G(x)=1)$ and  Geronimus $(W_C(x)=1)$ transformations of a measure  $\mu$ are
	\begin{align*}
	\hat{\mu} (x)&:=W_C(x) \mu(x), & \hat{\Omega}&=\Omega.\\
	\check{\mu} (x)&:=\frac{\mu(x)}{W_G(x)}+\sum_{i=1}^{s}\sum_{l=0}^{n_i-1}(-1)^{l} \frac{\xi_{i,l}}{l!}\delta^{(l)}(x-q_i),
	& \check{\Omega}&=\Omega \cup \{q_1,q_2,\dots q_s\}.
	\end{align*}
\end{pro}
Notice that the Geronimus transformation of a measure does not need to be a measure.

\begin{pro}[Hankel case: linear spectral]
Whenever the polynomials $W_C(x)$ and $W_G(x)$ are coprime the composition of the above transformations
is the linear spectral transformation
\begin{align*}
\tilde{\mu}(x)&:=\widehat{(\check{\mu})}(x)
=\frac{W_C(x) \mu(x)}{W_G(x)}+\sum_{i=1}^{s}\sum_{l=0}^{n_i-1}(-1)^{l} \frac{\xi_{i,l}}{l!}W_G(x)\delta^{(l)}(x-q_i),&
\tilde{\Omega}&=\Omega \cup \{q_1,q_2,\dots q_s\}.
\end{align*}
\end{pro}

The Christoffel formulas found for the three types of transformations hold equally in the more general Schwartz kernel situation. We do not need to have a Hankel symmetry for the Gram matrix.

Why linear spectral? Because the behavior of the Markov function
\begin{align*}
\hat{C}_0(y)&= \int \frac{(x-r)\d \mu}{y-x}=(y-r)\int \frac{\d \mu}{y-x}-\int \d \mu(x)=(y-r)C_0(y)-h_0, \\
\check{C}_0(y)&= \int \frac{\d \mu}{(x-q)(y-x)}+\xi\int \frac{\delta(x-q)\d \mu}{y-x}=
\frac{1}{(y-q)}\left[\int \frac{\d \mu}{y-x}-\int \frac{\d \mu(x)}{x-q}\right]+\xi\frac{1}{y-q}\\&=\frac{C_0(y)-C_0(q)+\xi}{(y-q)},\\
\tilde{C}_0(y)&= \int \frac{(x-r)\d \mu}{(x-q)(y-x)}+\xi\int \frac{\delta(x-q)(x-r)\d \mu}{y-x}\\&=
\frac{(y-r)C_0(y)-(q-r)C_0(q)+(q-r)\xi}{(y-q)}. \\
\end{align*}

See \cite{bueno-marcellan1,bueno-marcellan2} for a discussion of perturbations of bilinear forms and \cite{grunbaum-haine} for the study of Darboux transformations and orthogonal polynomials.

\subsection{Christoffel perturbations}

We began with an example. Take $W_C(x)=x-a$ for $a\in\C$, then the Gram matrices satisfy
$\hat G=(\Lambda-a)G$
and, therefore, the so called \textbf{connector} $\hat{\omega}:=\hat S_1(\Lambda-a)(S_1)^{-1}$
fulfills
$
\hat{\omega}:=\hat{S}_1(\Lambda-a)(S_1)^{-1}=\hat{H} \left(S_2 (\hat{S}_2)^{-1} \right)^{\top} H^{-1}$.
The \textbf{connector links the original and transformed orthogonal polynomials}
\begin{align*}
{\omega}P_1(x)&=(x-a)\hat{P}_1(x), &
\big(\hat  P_2(y) \big)^\top \hat H^{-1} \omega &=\big(P_2(y)\big)^\top H^{-1},
\end{align*} 
and is an upper diagonal matrix with only the first superdiagonal non-zero
\begin{align*}
{\omega}&=\begin{pmatrix}
{\omega}_{0,0} & 1 &     0     &    0                 &                          &                           \\
0         & {\omega}_{1,1} &   1&        0                 &  \ddots                       \\
0         &         0          &  {\omega}_{2,2}               &            1    &\ddots            \\
&                    &              \ddots     &                 \ddots        &     \ddots                                                     
\end{pmatrix}, & \omega_{k,k}&=\frac{\hat H_k}{H_k}.
\end{align*}

\begin{pro}The formula
	\begin{align*}
\omega_{n,n}=-\frac{P_{1,n+1}(a)}{P_{1,n}(a)}=\frac{\hat H_n}{H_n}
\end{align*}
holds.
\end{pro}

\begin{proof}
	It follows from 
$\omega_{n,n}P_{1,n}(x)+P_{1,n+1}(x)=(x-a)\hat P_{1,n}(x)$.
\end{proof}

\begin{pro} The CD kernels are connected by
	\begin{align*}
	(x-a)\hat K_{n}(x,y)-\hat P_{2,n}(y)(\hat H_{n})^{-1}P_{1,n+1}(x)=K_n(x,y).
	\end{align*}
\end{pro}

\begin{proof}
		As $\omega$ is upper triangular, from $	\big(\hat  P_2(y) \big)^\top \hat H^{-1} \omega =\big(P_2(y)\big)^\top H^{-1}$ we get
	\begin{align*}
	\big(\hat P_2^{[n+1]}(y)\big)^\top(\hat H^{[n+1]})^{-1}	\omega^{[n+1]}P_1^{[n+1]}(x)=\big(P_2^{[n+1]}(y)\big)^\top(H^{[n+1]})^{-1}P_1^{[n+1]}(x)=K_{n}(x,y).
	\end{align*}
	Observe also that
	\begin{align*}
	\omega^{[n+1]}P_1^{[n+1]}(x))&=(\omega P_1(x))^{[n+1]}-\begin{pmatrix}
	0_{n\times 1}\\ P_{1,n+1}(x)
	\end{pmatrix}= (x-a)\hat P_1^{[n+1]}(x)-\begin{pmatrix}
	0_{n\times 1}\\ P_{1,n+1}(x)
	\end{pmatrix}.
	\end{align*}
\end{proof}

\begin{pro}
The Christoffel formulas are
\begin{align*}
\hat P_{1,n}(x)&=\frac{1}{x-a}\Big(P_{1,n+1}(x)-\frac{P_{1,n+1}(a)}{P_{1,n}(a)}P_{1,n}(x)\Big),&
\hat P_{2,n}(y)&=\frac{K_n(a,y)}{P_{1,n}(a)}H_n,&
\hat H_n&=	-\frac{P_{1,n+1}(a)}{P_{1,n}(a)}H_n.
\end{align*}
\end{pro}

The Christoffel formulas in terms of quasi-determinants are
\begin{align*}
\hat P_{1,n}(x)&=\frac{1}{x-a}\Theta_*\begin{pmatrix}
P_{1,n}(a)&P_{1,n}(x) \\
P_{1,n+1}(a)&P_{1,n+1}(x)
\end{pmatrix}, &
\hat P_{2,n}(y)&=-\Theta_*\begin{pmatrix}
P_{1,n}(a)& H_n\\
K_n(a,y) &0
\end{pmatrix}, &
\hat H_n&=\Theta_*\begin{pmatrix}
P_{1,n}(a)& H_n\\
P_{1,n+1}(a) &0
\end{pmatrix}.
\end{align*}

In the Hankel case we have two alternative Christoffel type formulas. Indeed,
\begin{align*}
\hat P_{n}(x)&=\frac{1}{x-a}\Theta_*\begin{pmatrix}
P_{n}(a)&P_{n}(x) \\
P_{n+1}(a)&P_{n+1}(x)
\end{pmatrix}=-\Theta_*\begin{pmatrix}
P_{n}(a)& H_n\\
K_n(a,x) &0
\end{pmatrix}.
\end{align*}

As any polynomial $W_c(x)$ can be factorized in simple factors, the Christoffel formula  can be achieved by \textbf{iteration} of the previous simple example. 
We will take another path.

\begin{defi}[Jets]
	Given the set $r=\{(r_i,m_i)\}_{i=1}^{d}$  determined by the perturbing polynomial $W_C(x):=\prod_{i=1}^{d}(x-r_i)^{m_i}$ for any function $f(x)$ we define the corresponding jet 
\begin{align*}
\mathcal{J}_r[f]&:= 
\left( \frac{f^{(0)}(r_1)}{0!} , \frac{f^{(1)}(r_1)}{1!} , \dots , \frac{f^{(m_1-1)}(r_1)}{(m_1-1)!} ;
\frac{f^{(0)}(r_2)}{0!} , \dots , \frac{f^{(m_2-1)}(r_2)}{(m_2-1)!}; \dots ;
\frac{f^{(0)}(r_d)}{0!} , \dots , \frac{f^{(m_d-1)}(r_d)}{(m_d-1)!} 
\right).
\end{align*}
\end{defi}

\begin{pro}[General Christoffel formulas]
	The following Christoffel connection formulas hold
	\begin{align*}
\hat P_{1,n}(x)&=\frac{1}{W_C(x)}
\Theta_*
\begin{bmatrix}
\boldsymbol{\mathcal J}_{P_{1,n}} &P_{1,n}(x)\\ 	\vdots & \vdots  \\ \boldsymbol{\mathcal J}_{P_{1,n+N_C}}& P_{1,n+N_C}(x)
\end{bmatrix},\\
\hat H_{n}&=\Theta_*
\begin{bmatrix}
\boldsymbol{\mathcal J}_{P_{1,n}} &	H_{n}\\
\boldsymbol{\mathcal J}_{P_{1,n+1}} &0\\ 	\vdots & \vdots &\vdots \\ \boldsymbol{\mathcal J}_{P_{1,n+N_C}}& 0
\end{bmatrix},\\
\hat  P _{2,n}(y)&=-\Theta_*\begin{bmatrix} 	\boldsymbol{\mathcal J}_{ P_{1,n}}& H_{n}\\
\boldsymbol{\mathcal J}_{ P_{1,n+1}}&
0\\ 	\vdots & \vdots\\\boldsymbol{\mathcal J}_{P_{1,n+N_C-1}}&0\\
\boldsymbol{\mathcal J}_{K_{n-1}(\cdot,y)}(y)  &0
\end{bmatrix}.
\end{align*}
\end{pro}

\paragraph{\textbf{Comments}}

\begin{itemize}
\item The jets appear because the multiplicities bigger than 1 of each root
\item The idea of the proof is construct, as previously, a connector $\omega$ which is an upper triangular matrix with only the first $N_C$ superdiagonals nonzero
\item To get the transformations of the first polynomials $P_{1,n}(x)$ we use the connection formula in terms of the connector, evaluate in its zeroes, taking into account multiplicities and find the connector coefficients
\item To get the second polynomials $P_{2,n}$ we derive formulas for the Christoffel--Darboux kernels and apply a similar reasoning as in the example 
\item For the Hankel reduction we obtain  two alternative Christoffel type formulas for the polynomials
\end{itemize}

The previous formula reduces to the well known Christoffel formula.

\begin{pro}[Classical Christoffel formula]
\small\begin{align*}
\hat P_n(x)=\frac{1}{(x-r_1)\cdots(x-r_N)}\frac{\begin{vmatrix}
P_n(r_1) & \dots &P_n(r_N)&P_n(x)\\
\vdots &    &\vdots&\vdots\\
P_{n+N}(r_1) & \dots &P_{n+N}(r_N)&P_{n+N}(x)
\end{vmatrix}}{\begin{vmatrix}
P_n(r_1) & \dots &P_n(r_N)\\
\vdots &    &\vdots\\
P_{n+N-1}(r_1) & \dots &P_{n+N-1}(r_N)
\end{vmatrix}}.
\end{align*}
\end{pro}

\subsection{Geronimus transformations}
Take $W_G(x)=x-a$ for $a\in\C$,  with the perturbed Schwartz kernel given by
$\check u_{x,y}=\frac{u_{x,y}}{y-a}+\xi_x\delta_{y-a}$ 
with $\delta_y$ being the Dirac delta distribution and $\xi_x$ is a free linear functional. Then, the Gram matrices satisfy $\hat G	(\Lambda^\top-a)=G$ 
and, therefore, the so called \textbf{connector} $\omega:=\check S_1(S_1)^{-1}$
fulfills
${\omega}:=\check S_1(S_1)^{-1}=\check{H} (\check S_2)^{-\top} (\Lambda^\top-a)(S_2)^{\top}  H^{-1}$.

The \textbf{connector links the original and transformed orthogonal polynomials}
\begin{align*}
\check P_1(x) &=\omega P_1(x), &
\big(\check H^{-1}\omega H\big)^\top \check P_2 (y) &=P_2(y)(y-a),
\end{align*}
and is a lower diagonal matrix with only the first subdiagonal non-zero
\begin{align*}
\omega&=\begin{pmatrix}
1 &    0 & 0 &\dots                            \\
\omega_{1,0} & 1 &0   \\
0       &     \omega_{2,1}   &1      &\ddots  \\
\quad\quad\ddots&\quad\quad\ddots&   \quad\quad\ddots   \\
\end{pmatrix}, & \omega_{n+1,n}&=\check H_{n+1}(H_n)^{-1} .
\end{align*}

\begin{pro}
		The Geronimus transformation of the second kind functions satisfies
\begin{align*}
(x-a)	\check C_{1}(x)-\begin{pmatrix}
\check H_0
\\
0\\
\vdots
\end{pmatrix}&=	\omega C_1(x), &
\big(\check C_2(x)\big)^\top\check H^{-1}\omega  &= \big(C_2(x)\big)^\top H^{-1}.
\end{align*}
\end{pro}
\begin{proof}
	If follows from
\begin{align*}
(z-a)	\check C_1(z)-\omega C_1(z)&=\prodint{\check P_1(x),\frac{1}{z-y}}_{\check u}(z-a)-
\prodint{\check P_1(x),\frac{1}{z-y}}_{ (y-a)\check u } \\
&=\prodint{\check P_1(x),\frac{z-a-(y-a)}{z-y}}_{\check u}=\prodint{\check P_1(x),1}_{\check u}\\
&=\begin{pmatrix}
\check H_0
\\
0\\
\vdots
\end{pmatrix}
\end{align*}
and
	 \begin{align*}
	\big(\check C_2(x)\big)^\top\check H^{-1}\omega  - \big(C_2(x)\big)^\top H^{-1}&=
	\prodint{\frac{1}{z-x},\check P_2(y)}_{\check u}\check H^{-1}\omega -	\prodint{\frac{1}{z-x}, P_2(y)}_{ u} H^{-1}\\
	&=	\prodint{\frac{1}{z-x},\big(\check H^{-1}\omega \big)^\top\check P_2(y)}_{\check u}-	\prodint{\frac{1}{z-x},H^{-\top} P_2(y)}_{ u} \\
	&=	\prodint{\frac{1}{z-x}, (y-a)H^{-\top}P_2(y)}_{\check u}-	\prodint{\frac{1}{z-x},H^{-\top} P_2(y)}_{ u}\\&=0.
	\end{align*}
\end{proof}

We can not evaluate directly in $x=a$ as that point belongs to the support of the perturbed functional
and the second kind function might not be defined there.

\begin{pro}
	For $n>0$, we have
	\begin{align*}
	\omega_{n,n-1}&=-\frac{C_{1,n}(a)-\prodint{\xi_x, P_{1,n}(x)}}{C_{1,n-1}(a)-\prodint{\xi_x, P_{1,n-1}(x)}},&
	\check H_0&=-(C_{1,0}(a)-\prodint{\xi_x, 1}).
	\end{align*}
\end{pro}
\begin{proof}
	It follows form
	\begin{align*}
	(z-a)\check C_{1}(z)&=(z-a)\prodint{\check P_1(x),\frac{1}{z-y}}_{\frac{u_{x,y}}{y-a}+\xi_x \delta_{y-a}}=(z-a)\prodint{\check P_1(x),\frac{1}{z-y}}_{\frac{u_{x,y}}{y-a}}+\prodint{\xi_x,\check P_1(x)}
	\\&=(z-a)\prodint{\check P_1(x),\frac{1}{z-y}}_{\frac{u_{x,y}}{y-a}}+\omega\prodint{\xi_x, P_1(x)}.
	\end{align*}

	Thus, we derive an expression without singularity problems at $z=a$
	\begin{align*}
	(z-a)\prodint{\check P_1(x),\frac{1}{z-y}}_{\frac{u_{x,y}}{y-a}}-\begin{pmatrix}
	\check H_0
	\\
	0\\
	\vdots
	\end{pmatrix}&=\omega \big(C_1(z)-\prodint{\xi_x, P_1(x)}\big).
	\end{align*}

\end{proof}
\begin{pro}
	For $n>0$, we have the following quasideterminantal expression
	\begin{align*}
	\check P_{1,n}(x)&=\Theta_*\begin{pmatrix}
	C_{1,n-1}(a)-\prodint{\xi_x, P_{1,n-1}(x)} & P_{1,n-1}(x)\\
	C_{1,n}(a)-\prodint{\xi_x, P_{1,n}(x)}  &P_{1,n}(x)
	\end{pmatrix},&
	\check H_{n}&=\Theta_*\begin{pmatrix}
	C_{1,n-1}(a)-\prodint{\xi_x, P_{1,n-1}(x)} & H_{n}\\
	C_{1,n}(a)-\prodint{\xi_x, P_{1,n}(x)}  &0
	\end{pmatrix}.
	\end{align*}
\end{pro}
\begin{proof}
	For $n>0$, the expression for $\omega_{n,n-1}$ implies
	\begin{align*}
	\check P_{1,n}(x)&=P_{1,n}(x)-\frac{C_{1,n}(a)-\prodint{\xi_x, P_{1,n}(x)}}{C_{1,n-1}(a)-\prodint{\xi_x, P_{1,n-1}(x)}}P_{1,n-1}(x),&
	\check H_n&=-\frac{C_{1,n}(a)-\prodint{\xi_x, P_{1,n}(x)}}{C_{1,n-1}(a)-\prodint{\xi_x, P_{1,n-1}(x)}}H_{n-1}.
	\end{align*}
\end{proof}

\begin{pro}{}
The following connection formulas for Christoffel--Darboux kernels	 
hold\begin{align*}
	\check{K}_{n-1}(x,y)&=
	(y-a)K_{n-1}(x,y)-
	\check P_{2,n}(y)\check H^{-1}_n \omega_{n,n-1}	P_{1,n-1}(x),
	\end{align*}
and for $n\geq 1$
	\begin{align*}
	(x-a)	\check K_{n-1}^{(\operatorname{mix})}(x,y)&=(y-a)K^{(\text{mix})}_{n-1}(x,y)-\check P_{2,n}(y)\check H^{-1}_n  \omega_{n,n-1}	C_{1,n-1}(x)+1.
	\end{align*}
\end{pro}
\begin{proof}
		It follows from the definition  of the kernels and the connection formulas
	\begin{align*}
	\check P_1(x) &=\omega P_1(x),&
	\big( \check P_2 (y)\big)^\top \check H^{-1}\omega &=\big(P_2(y)\big)^\top(y-a)H^{-1}
	\end{align*}
	 and
$(x-a)	\check C_{1}(x)-\begin{pmatrix}
\check H_0,
0,
\dots
\end{pmatrix}^\top=	\omega C_1(x)$.
\end{proof}
\begin{pro} It also holds that
	\begin{align*}
	\check P_{2,n}(y)=H_{n-1}\frac{(y-a) (K_{n-1}^{(\text{mix})}(a,y)-\prodint{\xi_x,K_{n-1}(x,y)})+1}{C_{1,n-1}(a)-\prodint{\xi_x,P_{1,n-1}(x)}}.
	\end{align*}
\end{pro}
\begin{proof}
	The mixed kernel $\check K_{n-1}^{(\text{mix})}(x,y)$ will have singularity problems at $x=a$. This issue can be handled as before with the aid of the CD kernel and we get
	\begin{align*}
	(y-a) (K_{n-1}^{(\text{mix})}(a,y)-\prodint{\xi_x,K_{n-1}(x,y)})+1&=\check P_{2,n}(y)(\check H_n)^{-1}\omega_{n,n-1}\big(C_{1,n-1}(a)-\prodint{\xi_x,P_{1,n-1}(x)}\big)\\&
	=\check P_{2,n}(y)( H_{n-1})^{-1}\big(C_{1,n-1}(a)-\prodint{\xi_x,P_{1,n-1}(x)}\big).
	\end{align*}
\end{proof}

\begin{pro}
	For $n>0$, the following formulas hold
	\begin{align*}
	\check P_{2,n}(y)&=-\Theta_*\begin{pmatrix}
	C_{1,n-1}(a)-\prodint{\xi_x, P_{1,n-1}(x)} & H_{n-1}\\
	(y-a) (K_{n-1}^{(\text{mix})}(a,y)-\prodint{\xi_x,K_{n-1}(x,y)})+1&0
	\end{pmatrix}.
	\end{align*}
\end{pro}

\begin{pro}
	For a  general Geronimus perturbation we find the general Christoffel formulas
	\begin{align*}
	\check P^{[1]}_{n}(x)&=
	\Theta_*\begin{pmatrix}\boldsymbol{\mathcal J}_{C^{[1]}_{n-N}}-\prodint{ P^{[1]}_{n-N}(x),(\xi)_x}\mathcal W& P_{n-N}^{[1]}(x)\\ 	\vdots&\vdots\\
	\boldsymbol{\mathcal J}_{C^{[1]}_{n}}-\prodint{ P^{[1]}_{n}(x),(\xi)_x}\mathcal W& P^{[1]}_n(x)
	\end{pmatrix},\\
	\check H_{n}&=
	\Theta_*\begin{pmatrix}	\boldsymbol{\mathcal J}_{C^{[1]}_{n-N}}-\prodint{ P^{[1]}_{n-N}(x),(\xi)_x}\mathcal W& H_{n-N}\\
	\boldsymbol{\mathcal J}_{C^{[1]}_{n-N+1}}-\prodint{ P^{[1]}_{n-N+1}(x),(\xi)_x}\mathcal W& 0\\	\vdots&\vdots\\
	\boldsymbol{\mathcal J}_{C^{[1]}_{n}}-\prodint{ P^{[1]}_{n}(x),(\xi)_x}\mathcal W& 0
	\end{pmatrix},\\
		\big(	\check P _{n}^{[2]}(y)\big)^\top&=	-\Theta_*
	\begin{pmatrix}
	\boldsymbol{\mathcal J}_{C^{[1]}_{n-N}}-	\prodint{ P^{[1]}_{n-N}(x),(\xi)_x}\mathcal W&	H_{n-N}\\
	\boldsymbol{\mathcal J}_{C^{[1]}_{n-N+1}}-	\prodint{ P^{[1]}_{n-N+1}(x),(\xi)_x}\mathcal W& 0\\
	\vdots &\vdots\\	
	\boldsymbol{\mathcal J}_{C^{[1]}_{n-1}}-	\prodint{ P^{[1]}_{n-1}(x),(\xi)_x}\mathcal W&0\\
	W(y)\big(	\boldsymbol{\mathcal J}_{ K^{(pc)}_{n-1}}(y)-
	\prodint{ K_{n-1}(x,y),(\xi)_x}\mathcal W\big)+\boldsymbol{\mathcal J}_{ \mathcal V}(y)&0
	\end{pmatrix}.
	\end{align*}
\end{pro}
Here $\mathcal W,\mathcal V\in \C^{N_G\times N_G}$ are  upper triangular matrices determined by $W_G(x)$ and its derivatives, and  $\mathcal V(x,y)$ is a polynomial constructed in terms of $W_G(x)$ and completely symmetric bivariate polynomials.

\end{document}